\documentclass[10pt,a4paper,preprint]{elsarticle}
\usepackage{amsmath}
\usepackage{amsthm}
\usepackage{url}
\newtheorem{thm}{Theorem}
\newtheorem{prop}{Property}
\newtheorem{define}{Definition}

\journal{Linear Algebra and its Applications}

\begin{document}

\begin{frontmatter}

\title{Minimum Vertex Covers and the Spectrum of the Normalized Laplacian on Trees}
\author[FUB]{Hao Chen\fnref{hao}\corref{cor1}}
\ead{haochen@math.fu-berlin.de}
\address[FUB]{Freie Universit\"at Berlin, Institut of Mathematics, Berlin, Germany}
\fntext[hao]{Author was a master student of Ecole Polytechnique (Palaiseau, France) when doing this work at Max-Planck-Institut f\"ur Mathematik in den Naturwissenschaften (Leipzig, Germany) as an internship.}
\cortext[cor1]{Corresponding author.}

\author[MiS,SFI]{J\"urgen Jost}
\ead{jost@mis.mpg.de}
\address[MiS]{Max Planck Institute for Mathematics in the Sciences, Leipzig, Germany}
\address[SFI]{Santa Fe Institute, Santa Fe, USA}
\begin{abstract}

We show that, in the graph spectrum of the normalized graph Laplacian on trees, the eigenvalue 1 and eigenvalues near 1 are strongly related to minimum vertex covers.

In particular, for the eigenvalue 1, its multiplicity is related to the size of a minimum vertex cover,
and zero entries of its eigenvectors correspond to vertices in minimum vertex covers;
while for eigenvalues near 1, their distance to 1 can be estimated from minimum vertex covers; and for the largest eigenvalue smaller than 1, the sign graphs of its eigenvectors take vertices in a minimum vertex cover as representatives.
\end{abstract}

\begin{keyword}
Tree \sep Graph Laplacian \sep Graph spectrum \sep Minimum vertex cover \sep Eigenvalue 1 \sep Sign graph \sep Normalized Laplacian
\MSC[2010] 05C50 \sep 05C70
\end{keyword}

\end{frontmatter}

\section{Introduction}

Spectral graph theory tries to deduce information about graphs from the graph spectrum. For example, from the spectrum of the normalized Laplacian that we will study in this paper, one can obtain the number of connected components from the multiplicity of the eigenvalue 0, the bipartiteness from its largest eigenvalue (which is at most 2), as well as the connectivity (how difficult it is to divide a connected graph into two parts) from its second smallest eigenvalue.

Since the normalized Laplacian contains information of random processes on graphs, the eigenvalue 1 of the normalized Laplacian is realized to be important, and a very high multiplicity of 1 is often observed \cite{Banerjee2008a,Banerjee2008b}. Some interpretations of this high multiplicity have been proposed \cite{Banerjee2008b,Jost2006}.

In this paper, we shall explore a new relationship between the structure of a graph and the eigenvalue 1. We will show that, for trees, the eigenvalue 1 and eigenvalues near 1 are related to the minimum vertex cover problem, a classical optimization problem in graph theory. 

More specifically, minimum vertex covers can be used to calculate the multiplicity of eigenvalue 1, and are included by the zero entries of eigenvectors of 1. One can also use minimum vertex covers to estimate the shortest distance between the eigenvalue 1 and the other eigenvalues. Furthermore, for eigenvectors of the largest eigenvalue smaller than 1, vertices in a minimum vertex cover play the role of representatives for the sign graphs. 
A spectral property is therefore linked to a combinatorial problem on graphs.

\section{Brief introduction to the normalized Laplacian}

We will study undirected simple graphs $G=(V,E)$. An edge connecting two vertices $u,v\in V$ is denoted by $uv$. If $uv\in E$, we say that $u$ is a neighbor of $v$ and write $u\sim v$. The degree of a vertex $v$ will be denoted by $\deg{v}$.

The normalized Laplacian, which maps $\mathcal{F}$, the set of real valued functions of $V$, into itself, is a discrete version of the Laplacian in continuous space. Let $f\in\mathcal{F}$, and $u$ be a vertex in a graph $G$, then the normalized Laplacian operator is defined by
\begin{equation*}
\mathcal{L}f(u)=f(u)-\frac{1}{\deg{u}}\sum_{v\sim u}f(v).
\end{equation*}
It can be represented in matrix form by
\begin{equation*}
\mathcal{L}(u,v)=\begin{cases}
             1 & \text{if}\ u=v\\
             -\frac{1}{\deg{u}} & \text{if}\ u\sim v\\
             0 & \text{otherwise}
             \end{cases}
\end{equation*}
Here we have assumed that there is no isolated vertex, i.e. all the vertices have a positive degree.

We have the following immediate results about the spectrum of a normalized Laplacian: The normalized Laplacian is positive and similar to a symmetric linear operator, therefore its eigenvalues are real and positive, and are within the interval $[0,2]$. We can label the eigenvalues in non-decreasing order as $0=\lambda_1\le\ldots\le\lambda_n\le 2$, where $n=|V|$ is the number of vertices (same below). $\lambda_1=0$ is always the smallest eigenvalue, whose eigenvectors are locally constant functions, so its multiplicity is the number of connected components of the graph. For a bipartite graph $G=(V_1\sqcup V_2,E)$, if $\lambda$ is in the spectrum, so is $2-\lambda$. Therefore, for a connected graph, the largest eigenvalue $\lambda_n$ indicates the bipartiteness. It equals $2$ if the graph is bipartite, smaller otherwise. A discrete version of Cheeger's inequality 
\begin{equation*}
\frac{h^2}{2}\le\lambda_2\le 2h
\end{equation*}
is a famous result in spectral graph theory. Here $h$ is the discrete Cheeger's constant indicating how difficult it is to divide a graph into two parts. If the graph is already composed of 2 unconnected parts, $\lambda_2=0$. 

Detailed proofs of these results can be found in \cite{Chung1997} and \cite{Grigoryan2009}.

Since a tree is a connected bipartite graph, we know from the above results that its spectrum is symmetric with respect to 1, and that $0$ and $2$ are simple eigenvalues, and respectively the smallest and the largest eigenvalue. Since a deletion of any edge of a tree will divide a tree into two parts, the discrete Cheeger's constant is at most $2/n$, and the second smallest eigenvalue can be estimated by the discrete Cheeger's inequality.

\section{Minimum vertex cover and some of its properties}

By ``deleting a vertex $v$ from the graph $G=(V,E)$'', we mean deleting the vertex $v$ from $V$ and all the edges adjacent to $v$ from $E$, and write $G-v$. By ``deleting a vertex set $\Omega$ from the graph $G$'', we mean deleting all the elements of $\Omega$ from $G$, and write $G-\Omega$.

\begin{define}[vertex cover]
For a graph $G=(V,E)$, a \emph{vertex cover} of $G$ is a set of vertices $C\subset V$ such that $\forall e\in E, e\cap C\ne\emptyset$, i.e. every edge of $G$ is incident to at least one vertex in $C$. A \emph{minimum vertex cover} is a vertex cover $C$ such that no other vertex cover is smaller in size than $C$.
\end{define}

The minimum vertex cover problem is a classical NP-hard optimization problem that has an approximation algorithm. The following property is obvious, and will be very useful.

\begin{prop}
A vertex set is a vertex cover if and only if its complement is an independent set, i.e. a vertex set such that no two of its elements are adjacent. 
\end{prop}

So the minimum vertex cover problem is equivalent to the maximum independent set problem. For bipartite graphs, K\"onig's famous theorem relates the minimum vertex cover problem to the maximum matching problem, another classical optimization problem.

\begin{define}[matching]
For a graph $G=(V,E)$, a \emph{matching} of $G$ is a subgraph $M$ such that $\forall v\in M, \deg{v}=1$, i.e. every vertex in $m$ has one and only one neighbor in $M$. It can also be defined by a set of disjoint edges. A \emph{maximum matching} is a matching $M$ such that no other matching of $G$ has more vertices than $M$.
\end{define}

\begin{prop}[K\"onig's Theorem]
In a bipartite graph, the number of edges in a maximum matching equals the number of vertices in a minimum vertex cover.
\end{prop}

It should be noticed that, in general, neither a minimum vertex cover nor a maximum matching is unique. We show in Figure 1 a very simple graph, where the two white vertices form a minimum vertex cover. The meaning of the color and the size of vertices will be explained later. 

\begin{figure}[ht]
  \centering
  \includegraphics[scale=0.25]{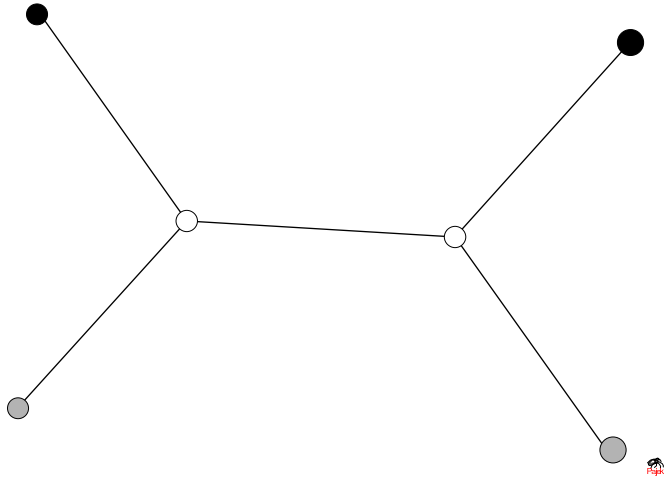}
  \caption{A simple graph. White vertices yield a minimum vertex cover.}
\end{figure}

We now prove the following properties of a minimum vertex cover, which will be useful for the proofs below.

\begin{prop}\label{PDeletion1}
Let $C$ be a minimum vertex cover of $G$. Then for any subset $C'\subset C$, $C-C'$ is a minimum vertex cover of $G-C'$.
\end{prop}
\begin{proof}
If $C-C'$ is not a minimum vertex cover of $G-C'$, there is a smaller vertex cover $C''$ of $G-C'$. Then $C''\cup C'$ covers every edge of $G$, and is smaller than $C$ in size. So $C$ is not a minimum vertex cover of $G$, contrary to the assumption. 
\end{proof}

\begin{prop}\label{PLeaves}
For a tree, let $L$ be the set of its leaves (vertices of degree 1). Then $L$ is not a subset of any minimum vertex cover.
\end{prop}
\begin{proof}
We'll argue by induction. The property is obviously true for a tree with less than 3 vertices. Suppose it's true for all trees with less than $n$ vertices. Now consider a tree $T$ of $n$ vertices.

Let $P$ be the set of parents (the only neighbors) of $L$. Assume a minimum vertex cover $C$ such that $L\subset C$. 

We have $P\cap C=\emptyset$. In fact, otherwise, consider $v\in P\cap C$, its child can be deleted from $C$, and the remaining vertex set is still a vertex cover, so $C$ is not a minimum cover as assumed.

Let $p$ be a mapping from $L$ to $P$ that maps a leaf to its only parent. We have $u\ne v\Rightarrow p(u)\ne p(v)$. In fact, otherwise, we can replace $u,v$ by their common parent in $C$, and the resulting vertex set is a vertex cover smaller than $C$, and the assumption is again violated.

So $p$ is a one-to-one map as long as the assumptions are true. We can replace $L$ by $P$ in $C$, and the resulting vertex set $C'$ is a vertex cover of size $|C|$, hence another minimum vertex cover. By the previous property, $C'-P$ is a minimum vertex cover of $T-P$. The leaves of $T-P$ are the grand parents of $L$.

By assumptions, there is at least one leaf $v'$ of $T-P$ that is not in $C'-P$, thus not in $C$. Neither is its child, because $P\cap C=\emptyset$. So the edge connecting $v'$ and his child is not covered by $C$. This violates the requirement that $C$ is a vertex cover. So the assumption that $L\subset C$ cannot be true for $T$.

By induction, $L\subset C$ is false for every tree.
\end{proof}

\begin{prop}\label{PExpand}
Let $C$ be a minimum vertex cover and consider a subset $C'\subset C$. Let $G'$ be the subgraph expanded by $C'$, that is, the vertex set of $G'$ consists of all the elements of $C'$ and all their neighbors, and the edge set of $G'$ consists of all the edges in $G$ that are adjacent to the elements of $C'$. Then $C'$ is a minimum vertex cover of $G'$.
\end{prop}
\begin{proof}
As in the proof of Property \ref{PDeletion1}, if $C'$ is not a minimum vertex cover of $G'$, we can find a vertex cover of $G$ smaller then $C$, thus violate the assumption.
\end{proof}


\begin{prop}\label{PDeletion2}
Let $z$ be a vertex excluded by every minimum vertex cover, then a minimum vertex cover $C$ of $G$ is also a minimum vertex cover of $G-z$.
\end{prop}
\begin{proof}
Consider a minimum vertex cover $C$. Let $N$ be the set of neighbors of $z$. Obviously, $N\subset C$. Assume a vertex cover $C'$ of $G-z$ smaller than $C$. If $N\subset C'$, $C'$ covers also all the edges of $G$, thus is a vertex cover of $G$ smaller than $C$, which is absurd. If $N\setminus C'$ is not empty, $C'\cup\{v\}$ is a vertex cover of $G$ of a size $\le|C|$, thus another minimum vertex cover $G$, which violates the assumption that $z$ is excluded by any minimum vertex cover of $G$.
\end{proof}

\section{Minimum Vertex Covers and Eigenvalue 1}
\subsection{Multiplicity of Eigenvalue 1}

In this part, we will show how to obtain the size of a minimum vertex cover of a tree from the multiplicity of the eigenvalue 1 of the normalized Laplacian. The vertices are labeled by integers $1,\ldots,n$ in non-decreasing order. We try to write out the characteristic polynomial of the normalized Laplacian matrix $\det(\mathcal{L}-xI)$. 

We use the expansion
\begin{equation*}
\det(A)=\sum_{\sigma\in S_n}\text{sgn}(\sigma)\prod_{i=1}^n A_{i,\sigma(i)}
\end{equation*}
where the sum is over all the permutations of vertices. Every permutation can be decomposed into disjoint cycles. A $k$-cycle with $k>2$ corresponds to a simple directed $k$-cycle in the graph, and a 2-cycle corresponds to an edge in the graph. Since the $(u,v)$ entry of $\mathcal{L}-xI$ is not zero if and only if there is an edge between $u$ and $v$, we conclude that a term in the summation above is not zero if and only if the permutation corresponds to a disjoint set (i.e. without common vertex) of cycles and edges in the graph. This is a fact noticed by many authors \cite[for example]{Mowshowitz1972}.

A tree is a graph without cycles, so every term of the characteristic polynomial corresponds to a set of disjoint edges, i.e. a matching. For a tree $T$, its characteristic polynomial of the normalized Laplacian can be written as
\begin{equation*}
P(x)=\sum_{M\in\mathcal{M}}\left((-1)^{|E_M|}(x-1)^{n-|V_M|}\prod_{v\in
M}\frac{1}{\deg{v}}\right) 
\end{equation*}
where $\mathcal{M}$ is the set of matchings $M=(E_M,V_M)$ of $T$. We see from this polynomial that the multiplicity of 1 is at least $\min_{\mathcal{M}}(n-|V_M|)$, or $n-|V_{\tilde{M}}|$, where $\tilde{M}$ is a maximum matching.

The characteristic polynomial can be further written as (with the convention that $0^0=1$)
\begin{eqnarray*}
P(x)&=&(x-1)^{n-|V_{\tilde{M}}|}\sum_{\mathcal{M}}\left((-1)^{|E_M|}(x-1)^{|V_{\tilde{M}}|-|V_M|}\prod_{v\in
M}\frac{1}{\deg{v}}\right)\\
&=&(x-1)^{n-2|V_{\tilde{C}}|}\sum_{\mathcal{M}}\left((-1)^{|E_M|}(x-1)^{2|V_{\tilde{C}}|-|V_M|}\prod_{v\in
M}\frac{1}{\deg{v}}\right)
\end{eqnarray*}
where $\tilde{C}$ is a minimum vertex cover. 

We see from this polynomial that, as long as the edge set is not empty, there will always be a matching, therefore the constant term of the sum will never vanish at $x=1$. So we have proved the following result:

\begin{thm}\label{mul1}
For a tree with a maximum matching $\tilde{M}$ and a minimum vertex cover $\tilde{C}$, the multiplicity of the eigenvalue 1 is exactly $n-|V_{\tilde{M}}|=n-2|V_{\tilde{C}}|$, i.e. the number of vertices unmatched by the maximum matching $\tilde{M}$.
\end{thm}

If we can find a minimum vertex cover of the tree, we know the size of a maximum matching by K\"onig's theorem, and then can tell the multiplicity of 1 as an eigenvalue of the normalized Laplacian.

\subsection{Eigenvalues near 1}

Let $\Lambda=\{\lambda_1,\ldots,\lambda_n\}$ be the spectrum of the normalized Laplacian for a tree $T$. We define the spectral separation by
$\bar{\lambda}=\min_{1\ne\lambda\in\Lambda}|1-\lambda|$, i.e. the shortest distance between the eigenvalue 1 and the other eigenvalues. In this part, we will give two upper bounds of this separation, using different methods, both taking advantage of properties of minimum vertex covers.

We now give the first upper bound of this separation, with a proof similar to the proof \cite{Grigoryan2009} of the second $\le$ of the discrete Cheeger's inequality ($\lambda_2\le 2h$).
Here, the measure of a vertex $v$ is defined by $\mu_v=\deg{v}$, while the measure (also known as the ``volume''\cite{Chung1997}) of a vertex subset $\Omega\subset V$ is defined by
\begin{equation*}
\mu_\Omega=\sum_{v\in\Omega}\mu_v=\sum_{v\in\Omega}\deg{v}.
\end{equation*}

\begin{thm}
$\bar{\lambda}\le\frac{\mu_{V-C}}{\mu_C}$ where $C$ is a minimum vertex cover of the tree in question.
\end{thm}

\begin{proof}
Let $0=\lambda_1\le\ldots\le\lambda_n=2$ be the eigenvalues, let $f_k$ be a $\lambda_k$-eigenvector. We know that 
\begin{equation*}
\lambda_k=\min_{f\in\Omega_{k-1}^{\perp}}\frac{\langle\mathcal{L}f,f\rangle}{\langle f,f\rangle}
\end{equation*}
where $\Omega_k=\{f_1,\ldots,f_k\}$.

Let $g\in\Omega_{n-|C|}^{\perp}$, i.e. $g$ is orthogonal to $f_1,\ldots,f_{n-|C|}$. The orthogonality gives a system of $n-|C|$ independent equations with $n$ unknowns, the dimension of the solution space is $|C|$. So we have the freedom to set $g$ to be a constant $a$ on $C$. We have, with the individual steps being explained subsequently,
\begin{eqnarray*}
\lambda_{n-|C|+1}&\le&\frac{\langle\mathcal{L}g,g\rangle}{\langle g,g\rangle}=\frac{\langle\nabla g,\nabla
g\rangle}{\langle g,g\rangle}\\
&=&\frac{\sum_{v\in V-C}(g(v)-a)^2\mu_v}{a^2\mu_C+\sum_{v\in
V-C}g(v)^2\mu_v}\\
&=&\frac{a^2\mu_{V-C}+2a^2\mu_C+\sum_{v\in V-C}g(v)^2\mu_v}{a^2\mu_C+\sum_{v\in
V-C}g(v)^2\mu_v}\\
&=&1+\frac{a^2\mu_V}{a^2\mu_C+\sum_{v\in V-C}g(v)^2\mu_v}\\
&\le&1+\frac{a^2\mu_V}{a^2\mu_C+a^2\frac{\mu_C^2}{\mu_{V-C}}}
=1+\frac{\mu_{V-C}}{\mu_C}
\end{eqnarray*}
which is the claim.

The second line is due to the fact that $g$ is constant on $C$, so if $v,w\in C$ and $v\sim w$, the edge $vw$ will not contribute in the calculation. Since $V-C$
is an independent set, we only need to consider the edges connecting $v\in C$ and $w\in V-C$.

The third line results from the orthogonality of $g$ to $f_1$ which is constant on $V$. This orthogonality implies that $a\mu_C+\sum_{v\in V-C}g(v)\mu_v=0$.

The last inequality comes from the relation
\begin{equation*}
\frac{\sum_{v\in
V-C}g(v)^2\mu_v}{\mu_{V-C}}\ge\left(\frac{\sum_{v\in
V-C}g(v)\mu_v}{\mu_{V-C}}\right)^2=\left(\frac{a\mu_C}{\mu_{V-C}}\right)^2
\end{equation*}
\end{proof}


Now, we will use the interlacing technique suggested by Haemers \cite{Haemers1995}, to find a second upper bound of the spectral separation.

\begin{define}[Interlacing]
Consider two sequences of real numbers $\lambda_1\le\cdots\le\lambda_n$ and $\mu_1\le\cdots\le\mu_m$ with $m<n$. The second sequence is said to \emph{interlace} the first one if $\lambda_i\le\mu_i\le\lambda_{n-m+i}$, for $i=1,\ldots,m$.
\end{define}

The following interlacing theorem \cite{Haemers1995} will be useful for us.

\begin{thm}[Haemers]
Suppose that the rows and columns of the matrix
\begin{equation*}
\begin{pmatrix}
 A_{1,1} & \cdots & A_{1,n} \\
 \vdots & \ddots & \vdots \\
 A_{n,1} & \cdots & A_{n,n}
\end{pmatrix}
\end{equation*}
are partitioned according to a partitioning $X_1,\ldots,X_m$ of
$\{1,\ldots,n\}$ with characteristic matrix $\tilde{S}$, i.e. $\tilde{S}(i,j)=1
\text{ if }i\in X_j\text{, and } 0$ otherwise. We construct the \emph{quotient
matrix} $\tilde{B}$ whose entries are the average row sums of the blocks of
$A$, i.e.
\begin{equation*}
\left(\tilde{B}_{ij}\right)=\frac{1}{|X_i|}(\tilde{S}^TA\tilde{S})_{ij}
\end{equation*}
Then the eigenvalues of $\tilde{B}$ interlace the eigenvalues of $A$.
\end{thm}

In \cite{Haemers1995}, a matrix is often partitioned into two parts in order to apply this theorem. Things will be complicated if we try to work with more parts. But, because of some properties of vertex covers, it is possible to partition a normalized Laplacian matrix into $n-|C|+1$ parts. We now prove our second upper bound of the separation.

\begin{thm}
\begin{equation*}
\bar{\lambda}\le 1-\frac{1}{|C|}\sum_{C\ni u\sim v\in C}\left(\frac{1}{\deg{u}}+\frac{1}{\deg{v}}\right)
\end{equation*}
where $C$ is a minimum vertex cover of the tree in question.
\end{thm}

\begin{proof}
Let's label the vertices in $V-C=\{v_1,\ldots,v_{n-|C|}\}$. We can now
partition the normalized Laplacian matrix into $n-|C|+1$ parts, by setting $X_0=C$ and $X_i={v_i}$. So the quotient matrix 
\begin{equation*}
\tilde{L}=\left[\begin{matrix}
                A&B\\
                C&I_n
                \end{matrix}
\right]
\end{equation*}
where the lower right part is $I_n$ because $V-C$ is an independent set. $B$ is a row matrix whose $i$-th entry is $\frac{1}{|C|}\sum_{u\sim v_i}\frac{-1}{\deg{u}}$. $C$ is a column matrix whose entries are all -1 (because the row sum of $\mathcal{L}$ is always 0). $A$ is a number whose value is $1-\frac{1}{|C|}\sum_{C\ni u\sim v\in C}\left(\frac{1}{\deg{u}}+\frac{1}{\deg{v}}\right)$

We know that
\begin{equation*}
\det\begin{pmatrix}A& B\\C& D \end{pmatrix}
=\det(D) \det(A - BD^{-1} C) 
\end{equation*}
so the characteristic polynomial of $\tilde{L}$ is 
\begin{equation*}
\det\begin{pmatrix}\
    \lambda-A & -B\\ -C & (\lambda-1)I_n
    \end{pmatrix}
=(\lambda-1)^{n-|C|}\left((\lambda-A)-\frac{1}{\lambda-1}BC\right)
\end{equation*}
Now that $V-C$ is an independent set, all neighbors of $v\in V-C$ are in $C$. So,
\begin{eqnarray*}
BC&=&\frac{1}{|C|}\sum_{v\in V-C}\sum_{u\sim v}\frac{1}{\deg{u}}
=\frac{1}{|C|}\sum_{C\ni u\sim v\in V-C}\frac{1}{\deg{u}}\\
&=&1-\frac{1}{|C|}\sum_{C\ni u\sim v\in C}\left(\frac{1}{\deg{u}}+\frac{1}{\deg{v}}\right)
=A.
\end{eqnarray*}
In fact, this is obvious because the row sums of $\mathcal{L}$ are zero, and so are the row sums of $\tilde{L}$.

So the eigenvalues of $\tilde{L}$ are $0\le 1\le 1+A$, where 0 and $1+A$ are simple eigenvalues, and 1 is an eigenvalue of multiplicity $n-|C|-1$.

By interlacing, we know that $\lambda_{n-|C|+1}\le 1+A\le \lambda_{n}=2$. So $A$ is an upper bound of the separation $\bar{\lambda}$.
\end{proof}

The graph in Figure 1 can be taken as a simple example. Both estimations give $2/3$ as the upper bound. This is an exact result, because all the inequalities in the proofs above become equalities for this graph.



\section{Minimum Vertex Cover and 1-Eigenvectors}

We show in Figure 2 a typical 1-eigenvector. All the pictures in this paper showing a real-valued function $f$ on $V$ will use the size of a vertex $v$ to represent the absolute value of $f(v)$, and the color of a vertex to represent the sign (black for negative, gray for positive, and the white vertices represent the zeroes).

\begin{figure}[ht]
  \centering
  \includegraphics[scale=0.3]{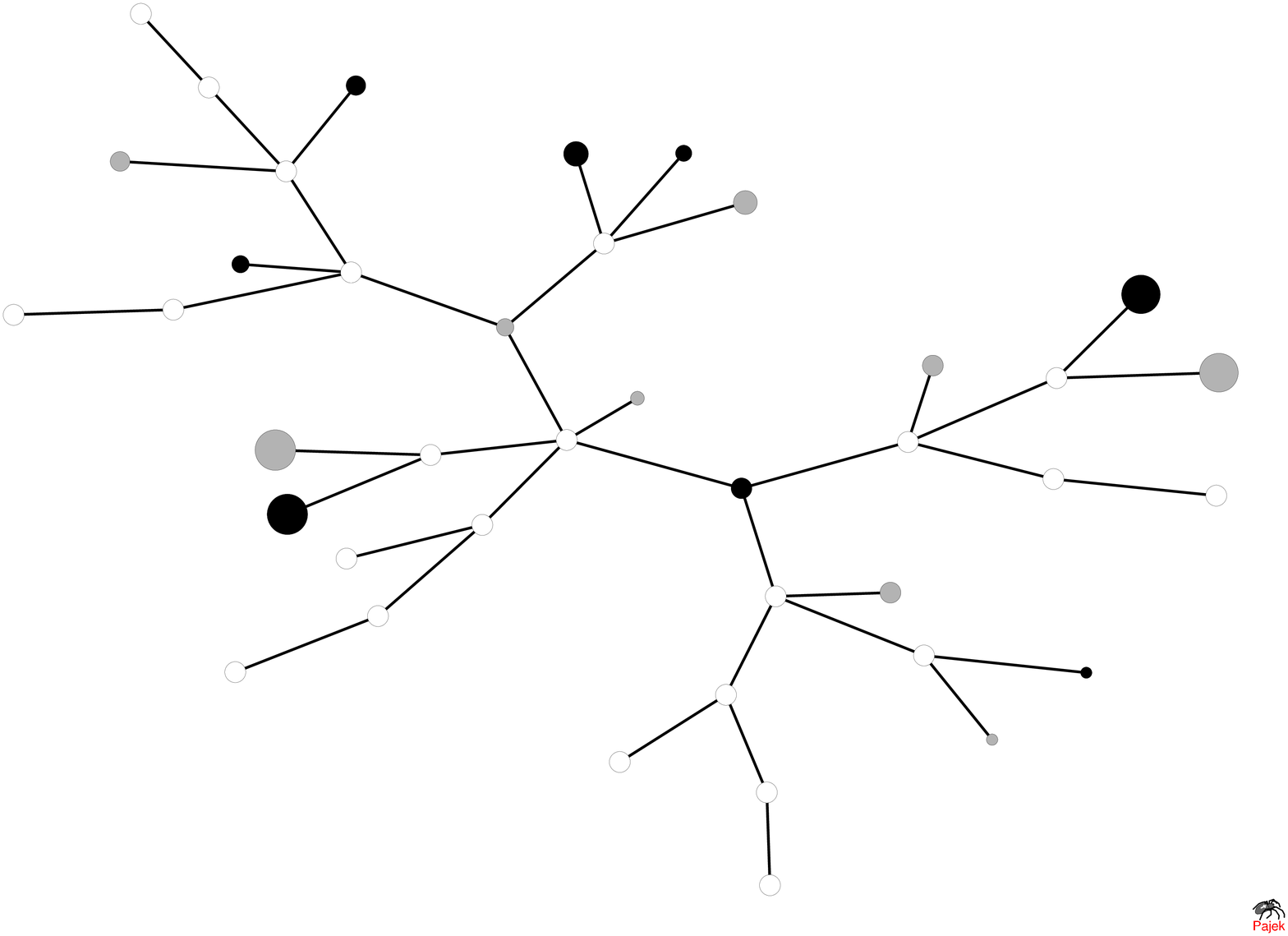}
  \caption{A typical 1-eigenvector.}
\end{figure}

It is not difficult to find a minimum vertex cover for such a small tree, and we find that a 1-eigenvector always vanishes (equals $0$) on a minimum vertex cover. This is more obvious in Figure 3 and Figure 1 (Figure 1 shows in fact a 1-eigenvector).

\begin{figure}[ht]
  \centering
  \includegraphics[scale=0.3]{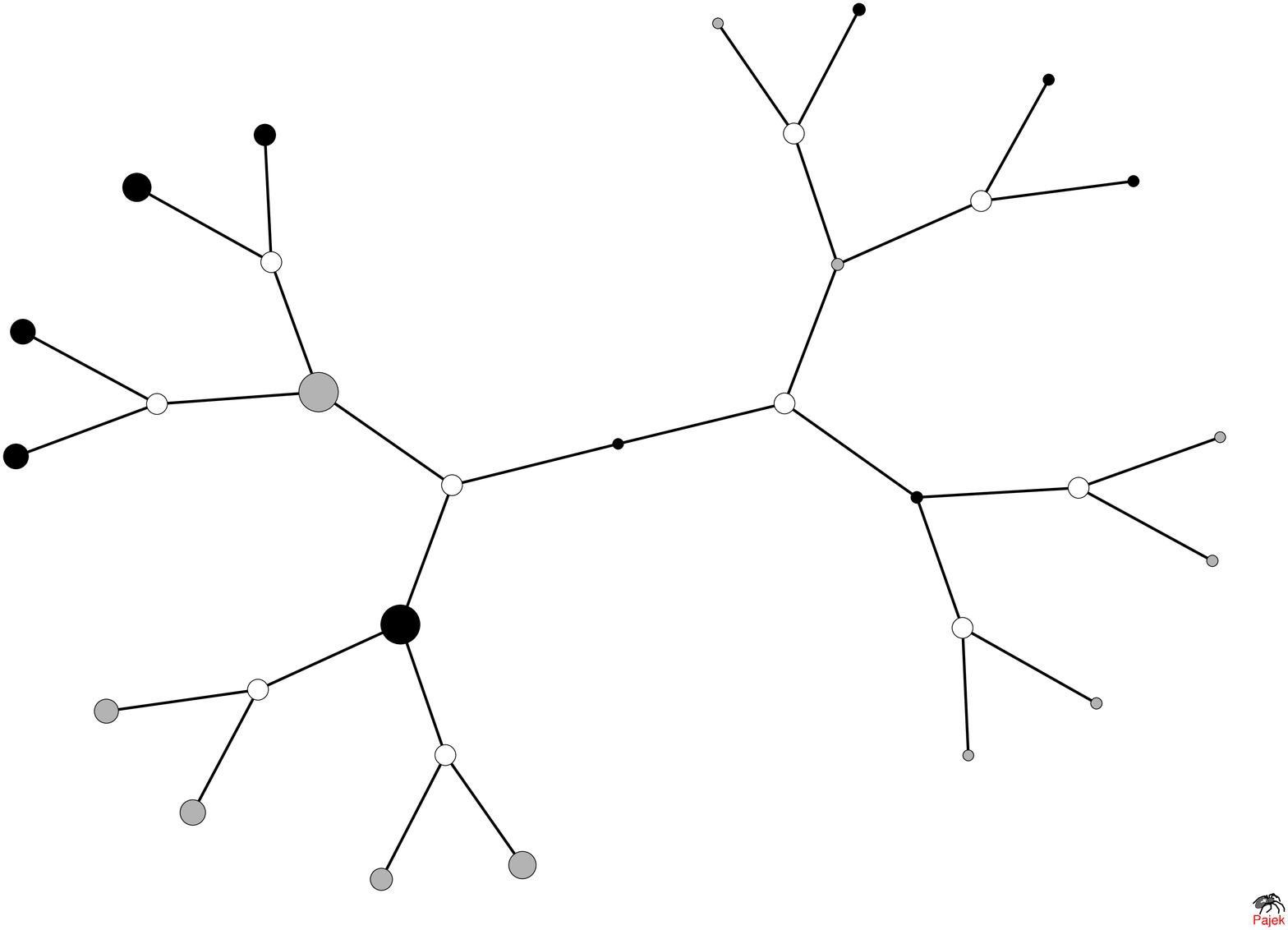}
  \caption{1-eigenvector of a symmetrical tree.}
\end{figure}

This observation is finally proved as the following theorem:

\begin{thm}\label{0pt1}
Let $T=(V,E)$ be a tree, $C$ be one of its minimum vertex covers, $f$ be one of its 1-eigenvectors, then $\forall c\in C, f(c)=0$. That is, any 1-eigenvector vanishes on all the minimum vertex covers. In other words, the set of vanishing points of any 1-eigenvector contains all the minimum vertex covers. 
\end{thm}

\begin{proof}
Since $C$ is a minimum vertex cover, its complement set $V-C$ is an
independent set. If $\forall c\in C, f(c)=0$, the Laplace equation for eigenvalue 1
\begin{equation*}
\mathcal{L}f(v)=f(v)-\frac{\sum_{u\sim v}f(u)}{\deg{v}}=f(v)
\end{equation*}
is automatically satisfied on $V-C$, since the average over their neighbors is 0.

In order to be a 1-eigenvalue, $f$ should satisfy for all vertices $c\in C$
\begin{equation*}
\sum_{v\sim c} f(v)=0
\end{equation*}
This is a system of $|C|$ linear equations with $n-|C|$ unknowns. By Properties \ref{PLeaves} and \ref{PExpand} of minimum vertex covers, these equations are independent.

Let $M$ be a maximum matching of a tree with $n$ vertices. It is obvious that $M$ has at most $\left[\frac{n}{2}\right]$ edges, by K\"onig's theorem, $|C|\le\left[\frac{n}{2}\right]$. An alternative argument is that, since a tree is bipartite, each of the two parts is a vertex cover, but not necessarily minimum, so $|C|\le\left[\frac{n}{2}\right]$. In the case where $n=2|C|$, $f=0$ is the only solution, because $V-C$ is also a minimum vertex cover. In the case where $n>2|C|$, there are more unknowns than equations, the dimension of the solution space is $n-2|C|$, which is exactly the multiplicity of the eigenvalue 1 as we have proved.

So a basis $f_1,\ldots,f_{n-2|C|}$ of the solution space is also a basis of the 1-eigenspace, and a 1-eigenvector must be a linear combination of $f_1,\ldots,f_{n-2|C|}$. This proves that every 1-eigenvector vanishes on $C$.
\end{proof}

\section{Minimum Vertex Covers and pre-1-Eigenvectors}

Here, by abuse of language, we mean by ``pre-1-eigenvectors'' the eigenvectors of the largest eigenvalue smaller than 1.

The sign graph (strong discrete nodal domain) is a discrete version of Courant's nodal domain. 

\begin{define}
Consider $G=(V,E)$ and a real-valued function $f$ on $V$. A \emph{positive (resp. negative) sign graph} is a maximal, connected subgraph of $G$ with vertex set $V'$, such that $f|_{V'}>0$(resp. $f|_{V'}<0$).
\end{define}

The study of the sign graphs often deals with generalized Laplacians. A matrix $M$ is called a generalized Laplacian matrix of the graph $G=(V,E)$ if $M$ has non-positive off-diagonal entries, and $M(u,v)<0$ if and only if $uv\in E$. Obviously, a normalized Laplacian is a generalized Laplacian. 

A \emph{Dirichlet normalized Laplacian} $\mathcal{L}_{\Omega}$ on a vertex set $\Omega$ is an operator defined on $\mathcal{F}_{\Omega}$, the set of real-valued functions on $\Omega$. It is defined by
\begin{equation*}
\mathcal{L}_{\Omega}f=(\mathcal{L}\tilde{f})|_{\Omega}
\end{equation*}
where $\tilde{f}\in\mathcal{F}$ vanishes on $V-\Omega$ and equals $f$ on $\Omega$. It can be regarded as a normalized Laplacian defined on a subgraph with boundary conditions, and has many properties similar to those of the normalized Laplacian. A Dirichlet normalized Laplacian is also a generalized Laplacian.

Previous works \cite{Gladwell2001,Biyikoglu2004,Biyikoglu2007} have established the following discrete analogues of Courant's Nodal Domain Theorem for generalized Laplacians:

\begin{thm}
Let $G$ be a connected graph and $A$ a generalized Laplacian of $G$, let the eigenvalues of $A$ be non-decreasingly ordered, and $\lambda_k$ be an eigenvalue of multiplicity $r$, i.e.
\begin{equation*}
\lambda_1\le\cdots\le\lambda_{k-1}<\lambda_k=\cdots=\lambda_{k+r-1}<\lambda_{k+r}\le\cdots\le\lambda_n
\end{equation*}
Then a $\lambda_k$-eigenvalue has at most $k+r-1$ sign graphs.
\end{thm}

In addition \cite{Biyikoglu2003} has studied the nodal domain theories on trees and even obtain equalities. But we have to study two cases
 
\begin{thm}[B{\i}y{\i}koglu]\label{Turker1}
Let $T$ be a tree, let $A$ be a generalized Laplacian of $T$. If $f$ is a $\lambda_k$-eigenvector without a vanishing coordinate (vertex where $f=0$), then $\lambda_k$ is simple and $f$ has exactly $k$ sign graphs.
\end{thm}

\begin{thm}[B{\i}y{\i}koglu]\label{Turker2}
Let $T$ be a tree, let $A$ be a generalized Laplacian of $T$. Let $\lambda$ be an eigenvalue of $A$ all of whose eigenvectors have at least one vanishing coordinate. Then
\begin{enumerate}
\item Eigenvectors of $\lambda$ have at least one common vanishing coordinate.
\item If $Z$ is the set of all common vanishing points, $G-Z$ is then a forest with components $T_1,\ldots,T_m$. Let $A_1,\ldots,A_m$ be the restriction of $A$ to $T_1,\ldots,T_m$, then $\lambda$ is a simple eigenvalue of $A_1,\ldots,A_m$, and $A_i$ has a $\lambda$-eigenvector without vanishing coordinates, for $i=1,\ldots,m$. 
\item Let $k_1,\ldots,k_m$ be the positions of $\lambda$ in the spectra of $A_1,\ldots,A_m$ in non-decreasing order. Then the number of sign graphs of an eigenvector of $\lambda$ is at most $k_1+\ldots+k_m$, and there exists a $\lambda$-eigenvector with $k_1+\ldots+k_m$ sign graphs.
\end{enumerate}
\end{thm}

In this theorem, if $A$ is the normalized Laplacian, the $A_i$ in the second item are in fact the Dirichlet normalized Laplacians on $T_i$. 

We denote by $\lambda_p$ the largest eigenvalue of the normalized Laplacian smaller than 1. We are interested in its eigenvectors.

Figure 4 shows a typical $\lambda_p$-eigenvector.
\begin{figure}[ht]
  \centering
  \includegraphics[scale=0.3]{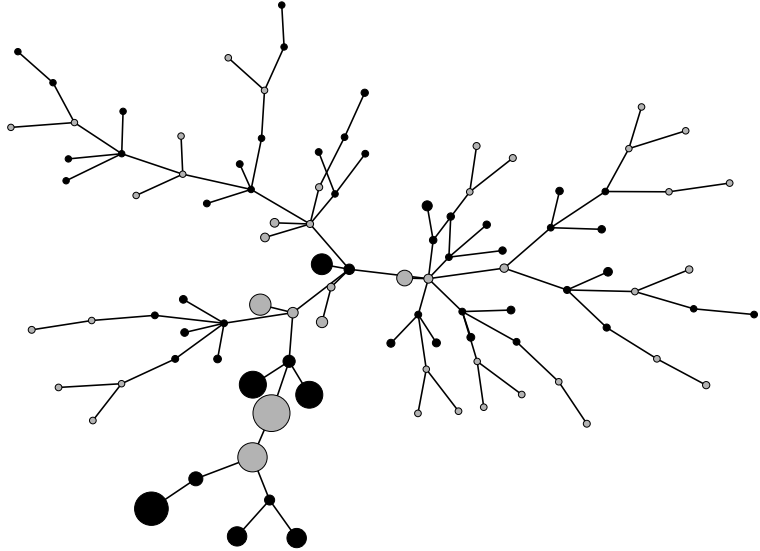}
  \caption{A typical $\lambda_p$-eigenvector.}
\end{figure}
We observe that vertices from a minimum vertex cover can be regarded as representatives for the sign graphs. This is also seen in Figures 5 and 6.
\begin{figure}[ht]
  \centering
  \includegraphics[scale=0.3]{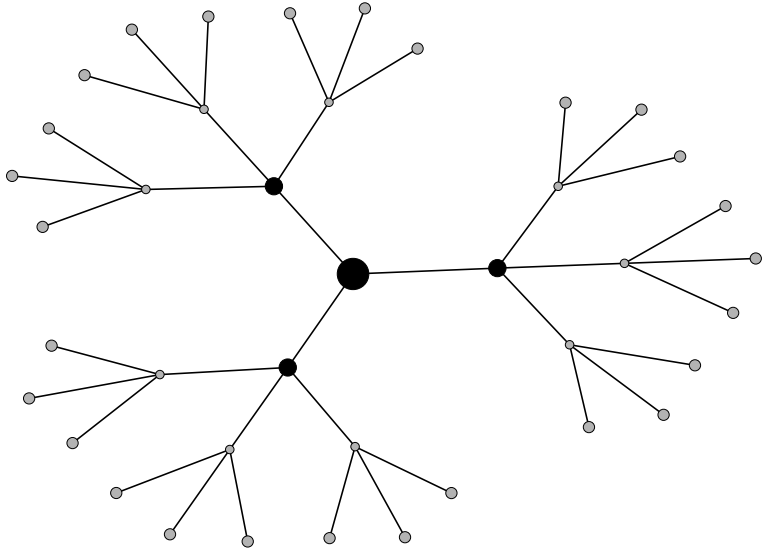}
  \caption{A typical $\lambda_p$-eigenvector without vanishing point.}
\end{figure}
\begin{figure}[ht]
  \centering
  \includegraphics[scale=0.3]{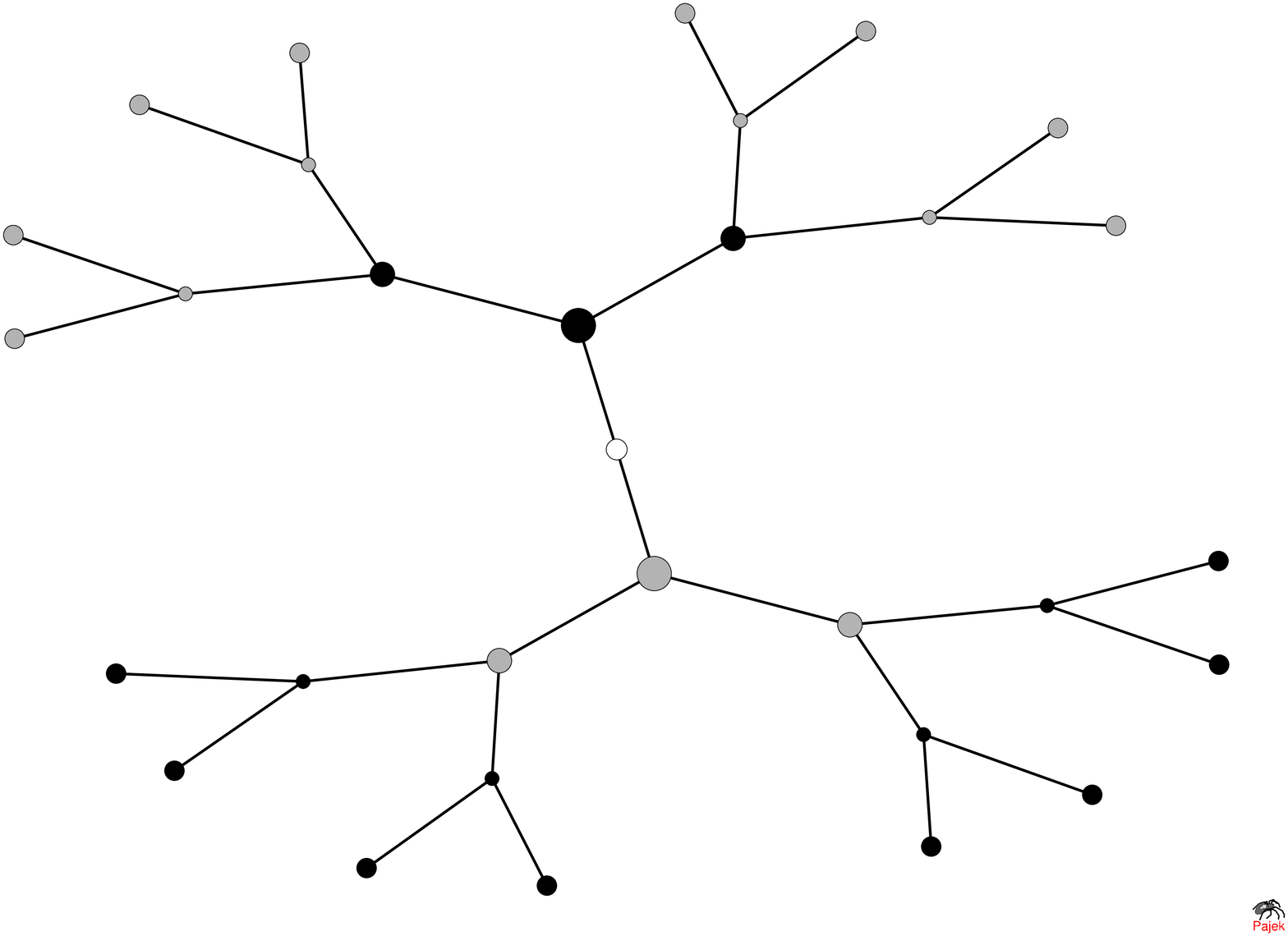}
  \caption{A typical $\lambda_p$-eigenvector with a vanishing point (in the middle).}
\end{figure}
With or without vanishing points, every sign graph contains one and only one vertex from a minimum vertex cover. 
As in \cite{Biyikoglu2003}, this observation is proved as a theorem by considering two different cases: with (Figure 6) or without (Figure 5) vanishing points.

\begin{thm}\label{SgnGrph1}
Let $T=(V,E)$ be a tree, let $C$ be a minimum vertex cover on $T$, we have $\lambda_p=\lambda_{|C|}$. If $f$ is a $\lambda_p$-eigenvector without vanishing coordinate, then $\lambda_p$ is a simple eigenvalue, and each of the $|C|$ sign graphs of $f$ contains \emph{one and only one} element $c\in C$, i.e. $C$ is a transversal of the sign graphs.
\end{thm}

\begin{proof}
It is immediate by the symmetry and Theorem \ref{mul1} that $\lambda_{|C|}<1$ and $\lambda_{|C|+1}\ge 1$. It is concluded
from B{\i}y{\i}koglu's Theorem \ref{Turker1} that $\lambda_p$ is simple, and a $\lambda_p$-eigenvector has $|C|$ sign graphs since it has no vanishing coordinate. We study the $\lambda_p$-eigenvector $f$.

We now prove that every sign graph of $f_p$ has at least 2 vertices. Otherwise, there will be a sign graph with only one vertex $v$, all of whose neighbors have an opposite sign, so $\mathcal{L}f(v)>f(v)$, which is not possible since $\lambda_p<1$.

We conclude that every sign graph contains at least one element of $C$, because $V-C$ is an independent set. Since there are exactly $|C|$ sign graphs, the only way to achieve this is to put exactly one element of $C$ in each sign graph.
\end{proof}

Now let's consider the case with vanishing coordinates, and prove the final theorem:

\begin{thm}\label{SgnGrph2}
Let $T=(V,E)$ be a tree, and $C$ be a minimum vertex cover.
\begin{enumerate}
  \item A $\lambda_p$-eigenvector has at most $|C|$ sign graphs, and there
  exists a $\lambda_p$-eigenvector with exactly $|C|$ sign graphs.
  \item Every sign graph of a $\lambda_p$-eigenvector contains one and only one
  element of $C$.
\end{enumerate}
\end{thm}

\begin{proof}
Only the case where all $\lambda_p$-eigenvectors have at least one vanishing coordinate remains to be proved. From Theorem \ref{Turker2}, we know that the $\lambda_p$-eigenvectors have at least one common vanishing coordinate.

Firstly, by the same method as for the normalized Laplacian, we can prove that Theorems \ref{mul1}, \ref{0pt1}, \ref{SgnGrph1} are also true for a Dirichlet normalized Laplacian.

As in the case without vanishing coordinates, we conclude from Theorem \ref{mul1} that $\lambda_p=\lambda_{|C|}$. Let $z$ be a common vanishing coordinate of the $\lambda_p$-eigenvectors, $\lambda_p$ is also an eigenvalue of the Dirichlet normalized Laplacian $\mathcal{L}_{T-z}$. The matrix form of $\mathcal{L}_{T-z}$ can be obtained by removing from $\mathcal{L}$ the row and the column corresponding to $z$, so the eigenvalues of $\mathcal{L}_{T-z}$ interlace the eigenvalues of $\mathcal{L}$ (see \cite{Haemers1995}).

We would like to prove that $z$ is not in any minimum vertex cover. Otherwise, assume a minimum vertex cover $C\ni z$. By Property \ref{PDeletion1}, $C-z$ is a minimum vertex cover of $T-z$. Applying Theorem \ref{mul1} to $\mathcal{L}_{T-z}$, we know that $\lambda^D_{|C|-1}<\lambda^D_{|C|}=1$, where $\lambda^D_1,\ldots,\lambda^D_{n-1}$ are the eigenvalues of $\mathcal{L}_{T-z}$ in non-decreasing order. By the interlacing argument, we conclude that $\lambda^D_{|C|-1}=\lambda_p$, and that the multiplicity of $\lambda_p$ in the spectrum of $\mathcal{L}_{T-z}$ is at most the same as in the spectrum of $\mathcal{L}$. 

This is however not possible if we look at the Laplacian equations with eigenvalue $\lambda_p$. After deleting the vertex $z$ from $T$, the Laplacian equation at vertex $z$ is eliminated from the equation system, thus the $\lambda_p$-eigenspace obtains one more dimension, which means that the multiplicity of $\lambda_p$ should be higher in the spectrum of $\mathcal{L}_{T-z}$ then in the spectrum of $\mathcal{L}$. Therefore, $z$ cannot be in any minimum vertex cover.

By Property \ref{PDeletion2}, $C$ is a minimum vertex cover of $T-z$. Let $z'$ be another common vanishing point of $\lambda_p$-eigenvectors of $\mathcal{L}$, it is obvious that it's also a common vanishing point of $\lambda_p$-eigenvectors of $\mathcal{L}_{T-z}$, so we can divide $T$ into a forest by deleting one by one all the common vanishing points, and finally conclude by applying Theorems \ref{SgnGrph1} and \ref{Turker2} to every single tree in the forest.
\end{proof}

Actually, this result is very intuitive. The minimum vertex covers try to cover the graph in a most efficient way, while the sign graphs try to divide the graph in a most uniform way. 





\section*{Acknowledgment}
We thank Frank Bauer for very helpful discussions, and the referee for his or her suggestions and careful review. The figures in this paper are generated by Pajek, a program for analyzing graphs.

\bibliographystyle{model1-num-names}
\bibliography{References}

\end{document}